\documentclass[a4paper,11pt]{article}
\usepackage{amsmath, amssymb, latexsym, amsthm, verbatim}
\usepackage[all]{xy}
\input{xy}

\long\def\forget#1\forgotten{}

\newcommand{\wzero}{\cup \{0\}}

\theoremstyle{plain} 
\newtheorem{thm}{Theorem}
\newtheorem{prop}[thm]{Proposition}

\newtheorem{lemma}[thm]{Lemma}
\newtheorem{cor}[thm]{Corollary}

\newtheorem{defn}[thm]{Definition}

\theoremstyle{definition}

\newtheorem{nota}[thm]{Notation}

\newtheorem{rem}[thm]{Remark}

\title{On the Cayley semigroup of a finite aperiodic semigroup}
\author{Avi Mintz\footnote{This research was done as part of the requirement for a PhD in Bar Ilan University and was done under supervision of Stuart Margolis}}

\begin{document}

\maketitle

\abstract{Let $S$ be a finite semigroup. In this paper we introduce
the functions $\varphi_s:S^* \to S^*$, first defined by Rhodes,
given by $\varphi_s([a_1,a_2 ,\ldots,a_n]) = [sa_1,sa_1a_2,\ldots,
sa_1a_2 \cdots a_n]$. We show that if $S$ is a finite aperiodic
semigroup, then the semigroup generated by the functions
$\{\varphi_s\}_{s \in S}$ is finite and aperiodic.

\section{Introduction}
Let $S$ be a finite semigroup. Rhodes considered in 1965
\cite{arbib} a function defined on the free monoid $S^*$. He called
this the machine of the semigroup. Formally, to each element $s$ of
the semigroup, he assigned a function $\varphi_s : S^* \to S^*$,
defined by $\varphi_s([a_1,a_2
,\ldots,a_n])=[sa_1,sa_1a_2,\ldots,sa_1a_2 \cdots a_n]$.
Essentially, the $\varphi_s$ arise from considering the Cayley graph
to be a sequential machine and assigning a function to each state $s
\in S$ considered as a start state. We call the semigroup generated
by all such functions the Cayley semigroup of $S$, and denote it by
$Cayley(S)$.

Rhodes showed that if $S$ is finite, then $S$ is aperiodic, that is,
has no non-trivial subgroups, if and only if for every $s \in S$, there
exists an $n \in \mathbb N$ such that $\varphi_s^n=\varphi_s^{n+1}$.
This construction played a key role in the original Krohn-Rhodes
Decomposition Theorem.


Grigorchuk and others have intensively studied semigroups and groups
generated by finite automata and these have many important
properties. Let $\mathbb Z_2$ denote the cyclic group of order $2$,
and denote its elements by $\{1,x\}$. The above gives two functions
$\varphi_1$ and $\varphi_x$ which are both invertible as functions
$\mathbb Z_2^* \to \mathbb Z_2^*$. A natural question is to describe
the group generated by them.

In \cite{grigplus2} and \cite{grigzuk} this group was proved to be
$\mathbb Z_2 wr \mathbb Z$, also known as the lamplighter group.
Later on this result was generalized by Silva and Steinberg to all
abelian group in \cite{benpedro} with some results on non-abelian
groups. Recently, Grigorchuk and Zuk, used these techniques to
calculate the spectrum of the lamplighter group \cite{grigzuk}.

In this paper we study the Cayley semigroups of other finite
semigroups. The first result is that if the semigroup $S$ divides
the semigroup $T$, then $Cayley(S)$ divides $Cayley(T)$. This
motivates us to work with semigroups that are not divisible by
groups, i.e. aperiodic semigroups. The main result, generalizing
Rhodes \cite{arbib} is that if $S$ is a finite aperiodic semigroup
then $Cayley(S)$ is a finite and aperiodic semigroup.

\section{Preliminaries}

Let $S$ be a semigroup. If $S$ is not a monoid, we denote by $S^1$
the monoid $S \cup \{1\}$ with the obvious multiplication making
$1$ the identity element. If $S$ happens to be a monoid already,
$S^1$ will equal $S$. In a similar way we can add a zero to a
semigroup $S$, if it doesn't have one. This semigroup is denoted by
$S^0$. For a set $X$ we write $X^*$ for the directed tree of all
strings over $X$, and $\epsilon$ for the empty word. For all $s \in
S$ define $\varphi_s:(S^1)^* \to S^*$ by $\varphi_s([a_1,a_2
,\ldots,a_n])=[sa_1,sa_1a_2,\ldots,sa_1a_2 \cdots a_n]$, and
$\varphi_s(\epsilon)=\epsilon$.
We note that unless $S$ is a group, these functions are not
invertible. We consider the semigroup generated by $\{\varphi_s\}_{s
\in S}$, and denote it by $Cayley(S)$. We will use
$\varphi_s,\varphi_t$ to denote generators of $Cayley(S)$ and $f,g$
to denote elements of $Cayley(S)$.

\subsection{Semigroup Actions}
A semigroup $S$ acts on the left of a set $Y$
if there is a correspondence $S \times Y \to Y$ $(s,y) \mapsto sy$
that satisfies $(st)y=s(ty)$.

A special case of this is when $Y$ has the structure of a rooted
tree. In our case we mean that every $y \in Y$ has a depth, denoted
by $|y|$ and every two elements $y_1,y_2 \in Y$ have a least upper bound
(lub) denoted by $y_1\vee y_2$, which is the longest common path on the
geodesics from the root to $y_1$ and $y_2$.

We assume the reader is familiar with these notions, but mention
that in the case of strings $X^*$ over an alphabet $X$ there is a
natural structure of a tree.

For example $|abbab|=5$ and $abbaab \vee abbb =abb$.

For two strings $w_1,w_2$, we call $w_1$ a prefix of $w_2$ if for
some $w_3$ we have $w_1 w_3=w_2$. Thus $u \vee v$ is the longest
common prefix of $u$ and $v$.

\begin{defn}
Let $f:X^* \to X^*$. We say that $f$ is a tree endomorphism if for
every $w_1,w_2 \in W^*$ we have
\begin{enumerate}
\item
$|w_1|=|f(w_1)|$,
\item
$f(w_1 \vee w_2)$ is a prefix of $f(w_1) \vee f(w_2)$.
\end{enumerate}
\end{defn}

%
%
\subsection{Mealy Automata}
We recall that a Mealy automata is a 5-tuple
$<A,Q,i,\delta,\lambda>$, with $A,Q$ finite sets, $i \in Q$,
$\delta:Q \times A \to Q$ and $\lambda:Q \times A \to A$. A Mealy
automata has a representation as a labeled directed graph. The
vertices of the graph are the elements of $Q$, and for every $q$ in
$Q$ and $a$ in $A$ there is an edge from $q$ to $\delta(q,a)$
labeled with $a/\lambda(q,a)$. For every word $w \in A^*$, we have a
walk on the graph, by starting at the vertex $i$ and each time we
see a letter $a$ of $w$ we walk from the state $q$ to the state
$\delta(q,a)$. Since this is a labeled graph, we can read the labels
as we read the word $w=[a_1,a_2,\ldots,a_n]$ and obtain the output
word
$f_i(w)=[\lambda(i,a_1),\lambda(\delta(i,a_1),a_2),\ldots,\lambda(\delta(\cdots(\delta(i,a_1),\cdots,)a_{n-1}),a_n)]$.
Thus a Mealy automata defines a function $f_i:A^* \to A^*$.

We can also ignore the state $i$, and thus the Mealy automaton
defines a collection of functions $A^* \to A^*$, each one determined
by its initial state $i$.

Let $S$ be a semigroup. The (right) Cayley graph of $S$ is the graph
described by a set of vertices $S^1$, and for every $s_1,s_2$ in $S$
there is an edge from $s_1$ to $s_1s_2$.

We can consider the Cayley graph of a finite semigroup as a Mealy
automaton by labeling the edge from $s_1$ to $s_1s_2$ with
$s_2/s_1s_2$. Using the formalities of Mealy automata this becomes
$Q=S$,$A=S^1$ and $\delta(q,a)=\lambda(q,a)=qa$.


\subsection{Examples}
Let's consider the five semigroups of order 2, and their Cayley
automata.

\begin{enumerate}
\item
$S_1=<a,b|a^2=ab=a,b^2=ba=b>$, the left zero semigroup of order $2$.

$
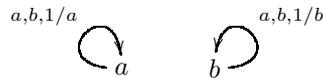
\xymatrix{a \ar@(l,u)^{a,b,1/a} &  b\ar@(r,u)_{a,b,1/b}} $

$\varphi_a([x_1,x_2,\ldots,x_n])=[ax_1,ax_1x_2,\ldots,ax_1x_2\cdots
x_n]=[a,a,\ldots,a]$.

We get $Cayley(S_1) \cong S_1$.

\item
$S_2=<a,b|a^2=ba=a,b^2=ab=b>$, the right zero semigroup of order
$2$.

$ 
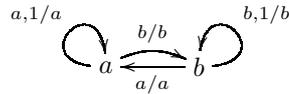
\xymatrix{a \ar@(l,u)^{a,1/a} \ar@/^/[r]^{b/b} &
b\ar@(r,u)_{b,1/b} \ar[l]^{a/a} }$

$\varphi_a([x_1,x_2,\ldots,x_n])=[ax_1,ax_1x_2,\ldots,ax_1x_2\cdots
x_n]$. If we write this as $\varphi_a([x_1,x_2,\ldots,x_n])=[y_1,y_2,\ldots,y_n]$, then we have $y_i=x_j$ where $j$ is the largest number less than $i$ such that $x_j\neq 1$. If no such $j$ exists, then $y_i=a$ (this $a$ comes from $\varphi_a$).

We verify that $\varphi_x\varphi_y = \varphi_y$ for $x,y \in
\{a,b\}$ so that $Cayley(S_2) \cong S_2$.

\item
$S_3=(\{0,1\},\cdot)$, the 2-element semilattice.

$ 
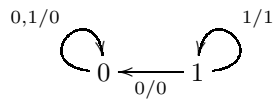
\xymatrix{0 \ar@(l,u)^{0,1/0} & 1\ar@(r,u)_{1/1} \ar[l]^{0/0} }$

$S_3$ is a monoid so we don't need to add a $1$ to the graph.

$\varphi_1([x_1,x_2,\ldots,x_n])=[1x_1,1x_1x_2,\ldots,1x_1x_2\cdots
x_n]$

$\varphi_0([x_1,x_2,\ldots,x_n])=[0,0,0,0,0]$

It is clear that $\varphi_0$ inputs a string of length $n$ and
outputs a string of $0$'s of length $n$.

Let $k$ be the smallest such that $x_k=0$. Then
$\varphi_1([x_1,x_2,\ldots,x_n])=[1,1,\ldots,1,0,0,\ldots,0]$ where
the first $0$ is at the $k$ place.

We get $Cayley(S_3) \cong S_3$.

\item
$S_4=<x|x^2=0>$, the 2 element nil semigroup.

$ 
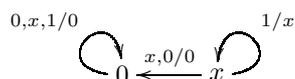
\xymatrix{0 \ar@(l,u)^{0,x,1/0} & x \ar[l]_{x,0/0}
\ar@(r,u)_{1/x}}$

$\varphi_x([x_1,x_2,\ldots,x_n])=[xx_1,xx_1x_2,\ldots,xx_1x_2\cdots
x_n]=[0,0,\ldots,0]$

$\varphi_x([1])=[x]$

$\varphi_0([x_1,x_2,\ldots,x_n])=[0,0,\ldots,0]$

We get $Cayley(S_4) \cong S_4$.

\item
$S_5=\mathbb Z_2=<x|x^2=1>$, the cyclic group of order 2.

$ 
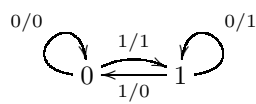
\xymatrix{0 \ar@(l,u)^{0/0} \ar@/^/[r]^{1/1} &  1\ar@(r,u)_{0/1}
\ar[l]^{1/0} }$

It is proved in \cite{benpedro} that $Cayley(S) \cong \{a,b\}^+$,
the free semigroup on two letters. In this case, $\varphi_x$ and
$\varphi_1$ are invertible and the group generated by
$\{\varphi_x,\varphi_1\}$ is isomorphic to the lamplighter group
$\mathbb Z_2 wr \mathbb Z$.
\end{enumerate}

Another interesting example is the following. Let $S$ be any monoid
and $1$ the identity element. In general,
$\varphi_1([a_1,a_2,\ldots,a_n]) = [a_1,a_1a_2,\ldots,a_1a_2\cdots
a_n]$ $\neq [a_1,a_2,\ldots,a_n]$, so the identity element of a
semigroup is not the identity element of $Cayley(S)$. Actually,
except for trivial cases, $Cayley(S)$ does not contain an identity
element.


\subsection{Finite Semigroups}

We assume standard terminology, facts and notation of finite semigroup theory and refer the reader to \cite{arbib} and \cite{cliffordpreston} for more details.

Let $A$ and $B$ denote finite sets. Let $G$ be a finite group. Let
$C:B \times A \to G \cup \{0\}$. We usually think of $C$ as a $B
\times A$ matrix. Furthermore, assume $C$ has the property that for
every $a \in A$ ($b \in B$) there is a $b \in B$ ($a \in A$) such
that $C(b,a)\neq 0$. A semigroup can be constructed, as follows,
from $A,B,G,C$, denoted by $\mathcal M^0(A,B,G;C)$ and called a Rees
matrix semigroup.

Let $\mathcal M^0(A,B,G;C) =(A \times G \times B) \wzero$ as sets. $0$
is a zero element. It remains to describe the product of elements in
$A \times G \times B$.

$(a_1,g_1,b_1)\cdot(a_2,g_2,b_2)=(a_1,g_1C(b_1,a_2)g_2,b_2)$ if
$C(b_1,a_2)\neq 0$ and $(a_1,g_1,b_1)\cdot(a_2,g_2,b_2)=0$
otherwise.

We state without proof the following theorem.

\begin{thm}[Rees-Sushkevitch]
A finite semigroup $S$ is $0$-simple if and only if it is
isomorphic to a Rees matrix semigroup.
\end{thm}

\begin{cor} \hfill\\
\begin{enumerate}
\item
A simple aperiodic semigroup is of the form $A \times B$, with
multiplication $(a_1,b_1)\cdot(a_2,b_2)=(a_1,b_2)$, and in
particular, is an idempotent semigroup.
\item
A $0$-simple aperiodic semigroup is of the form $(A \times B) \wzero$,
with multiplication defined by
$(a_1,b_1)\cdot(a_2,b_2)=(a_1,b_2)C(b_1,a_2)$, where $C$ is a $B \times A$ matrix over $\{0,1\}$.
\item
A non regular $\mathcal J$ class of a finite aperiodic semigroup is
of the form $A \times B$ with null multiplication.
\end{enumerate}
\end{cor}

\begin{proof} \hfill\\
\begin{enumerate}
\item
If $S$ is a simple aperiodic semigroup, then $S^0$ is a $0$-simple semigroup, and thus $S^0$ is isomorphic to $M=\mathcal M^0(A,B,G;C)$. Since $M$ has a subgroup isomorphic to $G$ (this can be shown by choosing $a,b$ such that $C(b,a)\neq 0$ and considering $\{(a,g,b)\}_{g \in G}$), and $M$ is aperiodic, we see that $G=\{1\}$. Since $S$ is a subsemigroup of $S^0$, the matrix $C$ does not contain any zeros.
\item
The proof of the second point is similar.
\item
This follows from the basic
theorems and definitions of the Sch\"utzenberger group. See
\cite{cliffordpreston}.
\end{enumerate}
\end{proof}

Another important corollary is the following.
\begin{cor}
In a $0$-simple semigroup, $s_1s_2\cdots s_n=0$ if and only if
$s_is_{i+1}=0$ for some $1\leq i \leq n-1$.
\end{cor}

We would like to extend the last two corollaries to a semigroup
acting on a $0$-minimal ideal. 

\begin{lemma}
Let $S$ be an aperiodic semigroup, and $I\cong A \times B \wzero$ a
$0$-minimal regular ideal. Let $s \in S$ and $(a,b) \in I \setminus
\{0\}$.
\begin{enumerate}
\item
$s(a,b)=0$ if and only if $s(a,b')=0$ for every $(a,b')$ in the
 $\mathcal R$ class of $(a,b)$.
\item
If $s(a,b)\neq 0$ then $s(a,b)=(a',b)$ for some $a' \in A$ which
depends only on $a$ and $s$.
\end{enumerate}
\end{lemma}
We will write this $a'$ as $sa$.

\begin{proof}
Both statements are a result of the fact that the $\mathcal R$
relation is a left congruence ($a \mathcal Rb$ implies $xa \mathcal
R xb$). 
For the second point, choose an $a_b \in A$ such that $C(b,a_b)=1$
and apply the associative law to $s(a,b)(a_b,b)$.
\end{proof}

\begin{nota}
This justifies an extended notation for the structure matrix of
$I$. $C:S \times A \to \{0,1\}$. $C(s,a)=0$ if $s(a,b)=0$ and
$C(s,a)=1$ if $s(a,b) \in A \times B$.
\end{nota}

\begin{nota}\label{nota sa}
Furthermore this shows that $S$ has a left action on the set $A \cup \{0\}$, where $sa$ is
the $a'$ in the formula $s(a,b)=(a',b)$. We have shown that this does not
depend on choice of $b$.
\end{nota}

With every $\mathcal J$ class of a semigroup $S$ we can associate a
$0$-simple or null semigroup, which we will now define.

\subsection{Trace of a $\mathcal J$ class}
Let $S$ be a semigroup, and $J$ a $\mathcal J$ class of $S$. Let
$\theta$ be an element disjoint from $S$ and consider the following
multiplication on $J$. If $j_1\cdot j_2$ isn't in $J$ we write their
product as $\theta$. The set $J \cup \{\theta\}$ with $\theta$ as a
zero element is a semigroup called the trace of $J$, and denoted by
$J^{tr}$. It is known (and easy to show) that $J^{tr} \prec S$.
Another possible construction of $J^{tr}$ is by defining $T$ to be
the ideal $S^1JS^1 \setminus J$ and then $J^{tr} \cong S^1JS^1 / T$.

\subsection{Semidirect product of semigroups}
Let $S$ and $T$ be finite semigroups. Suppose that $T$ acts on $S$
on the left. That is, for every $s \in S$ and $t \in T$ there is
an element of $S$ denoted by $^ts$ such that
$^{t_1}(^{t_2}s)=^{t_1t_2}s$ and $^t(s_1s_2)=^ts_1^ts_2$. In this
case we can define the (left) semidirect product of $S$ and $T$
whose underlying set is $S \times T$ with multiplication given by
$(s_1,t_1)(s_2,t_2)=(s_1^{t_1}s_2,t_1t_2)$. This semigroup is
denoted by $S \rtimes T$. There is also the dual notion of a right
semidirect product.

\section{Fractalness Property of $Cayley(S)$}

\subsection{The Pascal Array}
In \cite{arbib} Rhodes introduced the notion of a Pascal array. This
gives some intuition to the action of $Cayley(S)$. This construction
motivates many of the following proofs.

Let $s_1,s_2,\ldots,s_n\in S$ and $a_1,a_2,\ldots,a_k \in S^1$. We describe the
$(n+1) \times (k+1)$ table $(t_{ij})$. $t_{00}$ is left empty,
$t_{0i}=a_i$ and $t_{j0}=s_j$. The table is completed by the formula
$t_{mn}=t_{m(n-1)}\cdot t_{(m-1)n}$.

\hfill\\

\begin{tabular}{|r|r|r|r|r|}
\hline
&$a_1=t_{01}$ & $a_2=t_{02}$ & $a_3=t_{03}$ & $\cdots$\\
\hline
$s_1=t_{10}$ & $s_1a_1=t_{11}$ & $s_1a_1a_2=t_{12}$ & $s_1a_1a_2a_3=t_{13}$ & $\cdots$ \\
\hline
$s_2=t_{20}$ & $s_2s_1a_1=t_{21}$ & $s_2s_1a_1s_1a_1a_2=t_{22}$ & $s_2s_1a_1s_1a_1a_2s_1a_1a_2a_3=t_{23}$ & $\cdots$ \\
\hline
$s_3=t_{30}$ & $s_3s_2s_1a_1=t_{31}$ & $t_{31}t_{22}=t_{32}$ & $t_{32}t_{23}=t_{33}$ & $\cdots$ \\
\hline
$\cdots$ & $\cdots$ & $\cdots$ & $\cdots$ &\\
\hline
\end{tabular}

\hfill\\

The table describes the action of the function $\varphi_{s_n}\cdots
\varphi_{s_3}\varphi_{s_2}\varphi_{s_1}$ on the word
$[a_1,a_2,\ldots,a_k]$. The bottom row (excluding the first column) is the output.

\subsection{Fractalness}
Let $f=\varphi_{s_n} \ldots \varphi_{s_2} \varphi_{s_1} \in
Cayley(S)$ and let $v \in (S^1)^\star$. Since $Cayley(S)$ is a semigroup
of tree endomorphisms, the function $g:(S^1)^\star \to S^\star$
defined by $f(vw)=f(v)\circ g(w) \forall w \in (S^1)^\star$ is
well-defined. We denote this function $g$ by $f_v$.

The following theorem will show that $f_v$ is in $Cayley(S)$
whenever $f \in Cayley(S)$ and $v \in (S^1)^*$. Let $p:S^* \setminus
\{\epsilon\} \to S$ be the function defined by
$p([v_1,v_2,\ldots,v_n])=v_n$. Observe that for every $f \in
Cayley(S)$ and every $w \in (S^1)^*$ we have $pf(w) \in S$ (and
not\footnote{We do not identify strings of length 1 over $S$ with
$S$} in $S^*$). This shows that writing $\varphi_{pf(w)}$ is
meaningful.



\begin{thm}\label{fractalness}
Let $f\in Cayley(S)$. Let $v \in (S^1)^*$, and let $f_v:(S^1)^* \to S^*$
be defined by $f(vw)=f(v)\circ f_v(w)$. If we write $f$ as
$f=\varphi_{s_n} \ldots \varphi_{s_2} \varphi_{s_1}$ then
$f_v=\varphi_{(p\varphi_{s_n}\cdots\varphi_{s_1}(v))} \ldots
\varphi_{(p\varphi_{s_2}\varphi_{s_1}(v))}
\varphi_{(p\varphi_{s_1}(v))}$. We also have $f_\epsilon=f$. Thus if
$f \in Cayley(S)$, so is $f_v$ for all $v \in (S^1)^*$.
\end{thm}

\begin{proof}
This can be seen by observing the Pascal array. When we apply
$\varphi_{s_n}\cdots\varphi_{s_1}$ to $vw$, we can look at the
column of the last letter of $v$. The elements there are acting on $w$, describing the
action of $f_v$ on $w$. For a full proof we induct on $n$. If $n=1$
then let $w=[w_1,w_2,\ldots,w_k]$ and $v=[v_1,v_2,\ldots,v_l]$ and
write $\bar{v}=v_1v_2\ldots v_l$. Notice that $p\varphi_s(v)=s\bar
v$ for every $s$. We need to show that $\varphi_s(vw) =
\varphi_s(v)\circ\varphi_{p(\varphi_s(v))}(w)$.

Calculating gives us
\begin{align*}
\varphi_s(vw) &=
\varphi_s([v_1,v_2,\ldots,v_{l-1},v_l,w_1,\ldots,w_k]) \\
&= [sv_1,sv_1v_2,\ldots,sv_1v_2v_{l-1},s\bar v,s\bar v w_1,\ldots,s
\bar v w_1\cdots w_k] \\ &= [sv_1,sv_1v_2,\ldots,sv_1v_2v_{l-1},s\bar v]\circ[s\bar v w_1,\ldots,s
\bar v w_1\cdots w_k] \\&= \varphi_s(v) \circ\varphi_{s\bar v}(w) \\ &=
\varphi_s(v)\circ \varphi_{p({\varphi_s(v)})}(w)
\end{align*}

Let's assume this is correct for $n-1$. Let $v$ and $w$ be as
before.

Using the induction assumption (at the first $=$ sign) and the $n=1$
case (at the second), we get:
\begin{align*}
&\varphi_{s_n} \varphi_{s_{n-1}} \ldots \varphi_{s_1} (vw) = \\
&\varphi_{s_n} ( \varphi_{s_{n-1}} \ldots \varphi_{s_1} (v)\circ
\varphi_{(p\varphi_{s_{n-1}}\cdots \varphi_{s_1}(v))} \cdots \varphi_{(p\varphi_{s_2}\varphi_{s_1}(v))} \varphi_{(p\varphi_{s_1}(v))}(w) ) = \\
&\varphi_{s_n} ( \varphi_{s_{n-1}} \ldots \varphi_{s_1} (v))\circ
\varphi_{(p\varphi_{s_n}(\varphi_{s_{n-1}}\cdots\varphi_{s_1}(v)))}
\ldots \varphi_{(p\varphi_{s_2}\varphi_{s_1}(v))}
\varphi_{(p\varphi_{s_1}(v))}(w)) = \\
&\varphi_{s_n} \varphi_{s_{n-1}} \ldots \varphi_{s_1} (v)\circ
(\varphi_{(p\varphi_{s_n}\cdots\varphi_{s_1}(v))} \ldots
\varphi_{(p\varphi_{s_2}\varphi_{s_1}(v))}\varphi_{(p\varphi_{s_1}(v))})(w)
\end{align*}

\end{proof}

\section{Preliminary Results}

\subsection{Cayley as a Functor}
\begin{lemma}
\begin{enumerate}
\item
If $T$ is a quotient semigroup of $S$ then $Cayley(T)$ is a quotient
of $Cayley(S)$.
\item
If $T$ is a subsemigroup of $S$ then $Cayley(T) \prec Cayley(S)$.
\end{enumerate}
\end{lemma}

\begin{proof}
\begin{enumerate}
\item
Let $F:S\to T$ be a surjective morphism. We extend $F$ to $F:S^1 \to T^1$ and verify that this is a semigroup morphism. Write $\bar{s}$ for $F(s)$. Suppose we have in $Cayley(S)$ the following
equation $\varphi_s([a_1,a_2,\ldots,a_n]) = [sa_1 , sa_1a_2 , \ldots
, sa_1a_2 \cdots a_n]$. Then, applying $t\mapsto \bar{t}$ gives us
$\varphi_{\overline{s}}([\overline{a_1},\overline{a_2},\ldots,\overline{a_n}]) = [\overline{sa_1},\overline{sa_1a_2},\ldots,\overline{sa_1a_2
\cdots a_n}]$. This shows that the mapping $s\mapsto \bar{s}$ can be
extended to $\varphi_s \mapsto \varphi_{\bar{s}}$. We now extend
this mapping to $Cayley(S)\to Cayley(T)$. It remains to show that this is well defined, i.e. if $\varphi_{s_n}\cdots\varphi_{s_2}\varphi_{s_1}=\varphi_{r_m}\cdots\varphi_{r_2}\varphi_{r_1}$ in $Cayley(S)$ then
$\varphi_{\overline{s_n}}\cdots\varphi_{\overline{s_2}}\varphi_{\overline{s_1}} = \varphi_{\overline{r_m}}\cdots\varphi_{\overline{r_2}}\varphi_{\overline{r_1}}$ in $Cayley(T)$. This follows from the previous calculation

\item
Denote by $C$ the subsemigroup of $Cayley(S)$ generated by
$\{\varphi_t|t \in T\}$. An element of $C$ acts on the tree
$(S^1)^*$.

Assume first that $T=T^1$. In this case $(T^1)^*$ is a subtree of $(S^1)^*$
and is therefore invariant under the action of $C$, the map sending an element of
$C$ to its restriction on $(T^1)^*$ is an onto morphism $C\to Cayley(T)$.

We now consider the case $T\neq T^1$ and write $1_{T^1}$ for the identity of $T^1$. We can identify $1_{T^1}$ with $1_{S^1}$, and think of $T^1$ as a subsemigroup of $S^1$, so again $(T^1)^*$ is a subtree of $(S^1)^*$, so $Cayley(T)$ is a quotient of $C$.
\end{enumerate}
\end{proof}

The above proof actually shows that $Cayley$ is a functor,
i.e. if $\psi:S \to T$ is a semigroup morphism, then $\psi$ induces
a $Cayley(\psi):Cayley(S)\to Cayley(T)$.

\begin{cor}
Let $S$ and $T$ be semigroups. If $S \prec T$ then $Cayley(S) \prec
Cayley(T)$.
\end{cor}

\begin{rem}
Notice that if $S$ contains a zero element $0$, then $\varphi_0$ is
the zero element of $Cayley(S)$.
\end{rem}

\begin{rem}\label{generators are 1-1}
If $s\neq t$ then $\varphi_s \neq \varphi_t$.
\end{rem}

\begin{proof}
This is immediate by considering the action of both elements on the
word $[1]$.
\end{proof}

\subsection{Idempotent Semigroups}
Every $f \in Cayley(S)$ is a tree endomorphism, and therefore $f$ is
determined by its action on each node of the tree $(S^1)^*$ (where
the action is some function $S^1 \to S$). By \emph{the portrait}
of $f$ we mean the tree $(S^1)^*$, and for every node, a function
$S^1 \to S$. We mention an obvious but important fact, that two
elements of $Cayley(S)$ are equal if and only if the are decorated
with the same portrait. See \cite{grigplus2} for more on this
notation, and examples of how useful this is for proofs.

\begin{nota}
Denote by $l_s$ the function $l_s:S^1 \to S$ given by $l_s(t)=st$.
\end{nota}

Every element of $Cayley(S)$ is determined by its action
on the top level of the tree and by its restrictions to the
subtrees. Thus, every $f \in Cayley(S)$ can be represented by a pair $((f_{[s]})_{s \in S},l_s)$, where we recall the definition of $f_{[s]}$ as given in the fractalness section, and the fact that $f_{[s]} \in Cayley(S)$. If $f \in Cayley(S)$ is associated with the pair $((f_{[s]})_{s \in S},l_s)$, then $f([a_1,a_2,\ldots,a_n])=[l_s(a_1)]\circ f_{[sa_1]}([a_2,\ldots,a_n])$. It is easy to see that for generators, $\varphi_s \mapsto ((\varphi_{sx})_{x \in S},l_s)$, and that $(\varphi_t\varphi_s) \mapsto ((\varphi_{tsx}\varphi_{sx})_{x \in S},l_{ts})$.

We can then extend this to any product of generators,
$\varphi_{s_n}\ldots \varphi_{s_2} \varphi_{s_1} \mapsto ((\varphi_{s_n\cdots s_2s_1x}\cdots
\varphi_{s_2s_1x}\varphi_{s_1x})_{x \in S},l_{s_n \cdots
s_2s_1})$. 

\begin{lemma}
Let $S$ be an idempotent semigroup, Then $Cayley(S) \cong S$.
\end{lemma}

\begin{proof}




Let $s,t \in S$. Consider the word $w=[a_1,a_2,\ldots,a_n]$. We will show that $\varphi_t\varphi_s(w)=\varphi_{ts}(w)$.
Let $b_i=a_1a_2\cdots a_i$, and observe that since $S$ is an idempotent semigroup we have $sb_ib_j=sb_j$ when $i \leq j$ for every $s
\in S$.
We now calculate $\varphi_t\varphi_s(w) = \varphi_t([sa_1,sa_1a_2,\ldots,sa_1a_2\cdots a_n]) = \varphi_t([sb_1,sb_2,\ldots,sb_n]) = [tsb_1,tsb_1b_2,\ldots,tsb_1b_2\cdots b_n] = [tsb_1,tsb_2,\ldots,tsb_n] = \varphi_{ts}([a_1,a_2,\ldots,a_n])=\varphi_{ts}(w)$.

Since the action of the generators is faithful by Remark
\ref{generators are 1-1}, the mapping $s \mapsto \varphi_s$ is $1$-$1$
and unto $Cayley(S)$. This completes the proof.
\end{proof}

\begin{cor}
Let $S$ be a finite simple aperiodic semigroup. Then $Cayley(S)
\cong S$.
\end{cor}

\begin{proof}
As seen before, a simple aperiodic semigroup is of the form $A
\times B$, with $A$ a left zero semigroup, and $B$ a right zero
semigroup. (i.e. $(a_1,b_1)(a_2,b_2)=(a_1,b_2)$), and in particular
is an idempotent semigroup. By the above it immediately follows that
$Cayley(S) \cong S$.
\end{proof}

\subsection{Nilpotent Semigroups}
We recall here the definition of a nilpotent semigroup. We say that a semigroup
with zero is nilpotent of index $n$ if for every $m\geq n$ we have $s_m\cdots s_2s_1=0$
for every $s_m,\ldots,s_2,s_1 \in S$, and $n$ is the smallest with this property. The following observation shows that a finitely generated nilpotent semigroup is finite. If $S$ is such a semigroup, $X$ a
generating set and $n$ the nilpotency index, then $|S| \leq
\frac{|X|^n-1}{|X|-1}+1$ (unless $|X|=1$ in which case the semigroup has $n$ elements).

Here we will see another example where $S$ is similar to
$Cayley(S)$.

\begin{lemma}
If $S$ is a finite nilpotent semigroup of index $n$ then $Cayley(S)$
is nilpotent with nilpotency index at most $n$.
\end{lemma}

\begin{proof}
We use the wreath product notation. Let $S$ be a nilpotent semigroup
of index $n$. Let $\Pi_{i=1}^n \varphi_{s_i} = ((\Pi_{i=1}^n
\varphi_{b_i})_{x \in S},l_0)$ be a product of $n$ generators, for some $b_i$ who depend on $x$.
Since every $b_i$ is a product of at least one element of $S$, the
portrait is decorated by an $l_0$ in each node.

Since $S$ is finite,  $Cayley(S)$ is a finitely generated nilpotent
semigroup, and therefore is finite
\end{proof}

Recall that a monogenic semigroup is a semigroup which is generated
by one element.

\begin{cor}
If $S$ is a monogenic aperiodic semigroup, then $Cayley(S)$ is
nilpotent.
\end{cor}

A monogenic finite aperiodic semigroup is of the form
$\{x,x^2,\ldots,x^n\}$ satisfying the relation $x^n=x^{n+1}$ for
some $n\in \mathbb N$. Obviously, a monogenic semigroup is
commutative.

In general (i.e. $|S|$ not too small) if $S$ is monogenic, then
$Cayley(S)$ is not monogenic and not even commutative. For example
let $S=<x|x^5=x^6>$. In $Cayley(S)$ we have the following
computation.
$\varphi_x\varphi_{x^2}([1,1])=\varphi_x([x^2,x^2])=[x^3,x^5]$ and
$\varphi_{x^2}\varphi_x([1,1])=\varphi_{x^2}([x,x])=[x^3,x^4]$.

\subsection{Basic Structure}
We end this section with some structure facts about $Cayley(S)$.
Recall that the definition of a monoid acting on a set does not
require the identity of the monoid to have a trivial action.

\begin{thm}
Let $S$ be an aperiodic monoid, and $1$ the identity of $S$. The
following are equivalent.
\begin{enumerate}
\item
$Cayley(S)$ is a monoid, and the identity of $Cayley(S)$ acts
trivially on $S^*$.
\item
$\varphi_1$ has a trivial action on $S^*$.
\item
$S=\{1\}$.
\end{enumerate}
\end{thm}

\begin{proof}
$(3)$ implies $(2)$ implies $(1)$ is obvious. Assume that
$Cayley(S)$ has an identity which acts trivially on $S^*$. Let
$f=\varphi_{s_n}\cdots\varphi_{s_2}\varphi_{s_1}$ be the identity.
We have $[1]=f([1])=[s_n\cdots s_2s_1]$. This shows that $s_i
\mathcal J 1$, and in a finite aperiodic semigroup, this implies
$s_i=1$. We now have $f=\varphi_1^n$. Consider the word $[x,1]$.
$[x,1]=\varphi_1^n([x,1])=[x,x^n]$. Therefore, $x^n=1$, and by
aperiodicity $x=1$. This shows that $S=\{1\}$.
\end{proof}

\begin{thm}
Let $S$ be an aperiodic monoid, and $1$ the identity of $S$. The
following are equivalent.
\begin{enumerate}
\item
$\varphi_1$ is an idempotent.
\item
$\varphi_1$ is regular.
\item
$S$ is an idempotent monoid.
\end{enumerate}
\end{thm}

\begin{proof}
We saw that if $S$ is idempotent, then $S \cong Cayley(S)$. This
shows that $(3)$ implies $(1)$ and $(1)$ implies $(2)$. Now suppose
$\varphi_1$ is regular. For some $f \in Cayley(S)$,
$\varphi_1f\varphi_1=\varphi_1$. Considering the word $[1]$, we have
$\varphi_1f\varphi_1([1])=\varphi_1([1])=[1]$. If we write $f=\varphi_{s_n}\cdots\varphi_{s_2}\varphi_{s_1}$, then $\varphi_1 f\varphi_1([1])=[s_n\cdots s_2s_1]$, thus $s_n\cdots s_2s_1=1$. In an aperiodic semigroup, this implies $s_n=\cdots=s_2=s_1=1$, and so $f=\varphi_1^n$ for some $n$. This gives us
$\varphi_1=\varphi_1^{n+2}$.Applying both functions on $[x,1]$ gives
$[x,x]=[x,x^{n+2}]$. This gives $x=x^{n+2}$ and by aperiodicity,
$x=x^2$ for every $x \in S$.
\end{proof}

\begin{lemma}
$Cayley(S \times T)$ is a subsemigroup of  $Cayley(S) \times
Cayley(T)$.
\end{lemma}

\begin{proof}
$\Phi:Cayley(S \times T) \to Cayley(S) \times Cayley(T)$ defined by
$\Phi:\Pi_{i=1}^n\varphi_{(s_i,t_i)} \mapsto
(\Pi_{i=1}^n\varphi_{s_i},\Pi_{i=1}^n\varphi_{t_i})$ maps $Cayley(S
\times T)$ onto the subsemigroup of $Cayley(S) \times Cayley(T)$ of
pairs of functions which can be written as elements of the same
length. To see that this is well defined let $f \in Cayley(S \times T)$ map the string $w$ to $w'=f(w)$,
where $w,w' \in ((S \times T)^{1})^*$. If we denote by $w_S$ the string obtained from $w$ by replacing every pair
$(s_i,t_i)$ with $s_i$, then we can see that $\Phi(f)$ maps $w_S$ to $w'_S$, regardless of the presentation of $f$ as
a product of generators. The same argument works for $T$.

To see that this is a morphism, let $f=\Pi_{i=1}^n\varphi_{(s_i,t_i)}$, and $g=\Pi_{i=1}^m\varphi_{(s'_i,t'_i)}$. We have $\Phi(f)=(\Pi_{i=1}^n\varphi_{s_i},\Pi_{i=1}^n\varphi_{t_i})$, and $\Phi(g)=(\Pi_{i=1}^m\varphi_{s'_i},\Pi_{i=1}^m\varphi_{t'_i})$. Thus
\begin{align*}
&\Phi(f)\Phi(g)=\\
&(\Pi_{i=1}^n\varphi_{s_i},\Pi_{i=1}^n\varphi_{t_i}) (\Pi_{i=1}^m\varphi_{s'_i},\Pi_{i=1}^m\varphi_{t'_i}) = \\
&(\Pi_{i=1}^n\varphi_{s_i}\Pi_{i=1}^m\varphi_{s'_i},\Pi_{i=1}^n\varphi_{t_i}\Pi_{i=1}^m\varphi_{t'_i}) = \\ &\Phi(fg).
\end{align*}
\end{proof}

\section{The Action on Minimal Ideals}

We assume henceforth that all semigroups have an identity and a
zero. We do not lose any generality for the main results since
$Cayley(S) \prec Cayley(S^1)$ and $Cayley(S) \prec Cayley(S^0)$.

Let $S$ be a finite aperiodic semigroup, and $I$ a $0$-minimal ideal. We will show
that the restriction of $Cayley(S)$ to $I$ is finite and aperiodic.

\begin{nota}
$Cayley(S,I)$ will denote the restriction of $Cayley(S)$ to $I^*$, i.e.
$Cayley(S,I)$ is the semigroup of functions $I^* \to I^*$,
generated by the functions $\{\varphi_s\}_{s \in S}$.

More generally, $Cayley(S,T)$ will denote the restriction of
$Cayley(S)$ to the tree $T^*$ for some ideal $T$ of $S$.
\end{nota}

\subsection{The Rhodes Expansion}
Let $S$ be a finite semigroup. The right Rhodes expansion $\hat S$
was introduced by Rhodes in 1969 \cite{rhodes expansion}.

Consider the set $M(S)$ of all strings in $S^*$ of the form
$[s_1,s_1s_2,s_1s_2s_3,\ldots,s_1s_2\cdots s_n]$. Define a
multiplication on $M(S)$ by
$$[s_1,\ldots,s_1s_2\cdots s_n][t_1,\ldots,t_1t_2\cdots t_m]=[s_1,\ldots,s_1s_2\cdots s_n,s_1\cdots s_nt_1,\ldots,s_1\cdots s_nt_1 \cdots t_m]$$
Consider the case where $s_1\cdots s_{i-1}s_i \mathcal R s_1\cdots
s_{i-1}$. Equivalently, for some $x \in S$, $s_1\cdots s_{i-1}s_ix
=  s_1\cdots s_{i-1}$. A \emph{one step reduction} replaces \\
$w=[s_1,\ldots,s_1\cdots s_{i-1},s_1\cdots
s_{i-1}s_i,\ldots,s_1\cdots s_n]$ with \\$w' = [ s_1 , \ldots , s_1
\cdots s_{i-2} , s_1 \cdots s_{i-1} s_i , \ldots , s_1 \cdots s_n]$.
If no such $i$ exists, $w$ is called reduced. It is easy to see that
every element $w$ of $M(S)$ has a unique reduced form $red(w)$
obtained from $w$ by applying finitely many one step reductions.

\begin{defn}
The right Rhodes expansion of a finite semigroup $S$ is the
collection of reduced words in $M(S)$ with multiplication $w\cdot w'
= red(ww')$. We denote it by $\widehat S^R$.
\end{defn}

The Rhodes expansion is closely related to the following
construction. 

\subsection{The Memory Semigroup of $S$}

Let $S$ be a finite semigroup. Define a multiplication on the set $S
\times P(S)$ by
$$(s,\alpha)\cdot(t,\beta)=(st,\alpha t \cup \{t\} \cup \beta)$$
where, if $\alpha \in P(S)$ and $t \in S$, then $\alpha t=\{xt|x \in
\alpha\}$. We will first verify that this multiplication is
associative.

$[(s,\alpha)(t,\beta)](x,\gamma)=(st,\alpha t \cup \{t\} \cup
\beta)(x,\gamma)=(stx,\alpha tx \cup \{tx\} \cup \beta x \cup
\{x\} \cup \gamma)$\\
$(s,\alpha)[(t,\beta)(x,\gamma)]=(s,\alpha)(tx,\beta x \cup \{x\}
\cup \gamma)=(stx,\alpha tx \cup \{tx\} \cup \beta x \cup \{x\} \cup
\gamma)$.

\begin{defn}
We call the above the (right) memory semigroup of $S$, and denote it
by $mem(S)$.
\end{defn}

Note that $f:mem(S)\to S$,$f(s,\alpha)=s$ is a surjective morphism.
Thus $S \prec mem(S)$.

The finiteness of $mem(S)$ is trivial. Furthermore, we observe the
following.
\begin{lemma}
Let $S$ be a finite semigroup. Then $mem(S)$ is aperiodic if and
only $S$ is.
\end{lemma}

\begin{proof}
A simple induction shows that $(s,\alpha)^n=(s^n,\alpha \cup
\alpha\{s,s^2,\ldots,s^{n-1}\}\cup \{s,s^2,\ldots,s^{n-1}\})$. Thus
if $s^n=s^{n+1}$ for every $s$ in $S$ then
$(s,\alpha)^{n+1}=(s,\alpha)^{n+2}$ for every $(s,\alpha)$ in
$mem(S)$. The converse follows since $S \prec mem(S)$.
\end{proof}

We now proceed with the action of $Cayley(S)$ on $I^*$.

\subsection{Some Lemmas}

\begin{lemma} \label{almost simple}
Let $S$ be an aperiodic semigroup. Let $I \cong (A \times B) \wzero$
be a $0$-minimal ideal, $f\in Cayley(S,I)$, and $w \in I^* \setminus
\{\epsilon\}$.
\begin{enumerate}
\item
$f(w)$ is of the form $[(a,b_1),(a,b_2),\ldots,(a,b_k),0,\ldots,0]$,
with $a \in A$,$b_i \in B$ and $k\geq 0$.
\item
The first $k$ elements of $w$ are
$[(a_1,b_1),(a_2.b_2),\ldots,(a_k,b_k)]$, for some $a_i \in A$.
\item
Assume $k>0$. If we write
$f=\varphi_{s_n}\cdots\varphi_{s_2}\varphi_{s_1}$ then $a$ is given
by $s_n\cdots s_2s_1(a_1,b_1)=(a,b_1)$, that is, $s_n\cdots
s_2s_1a_1=a$, assuming $f(w)$ is not a string of zeros.
\end{enumerate}
\end{lemma}

\begin{proof}
\begin{enumerate}
\item
It is enough to show this for the case when $f$ is a generator of
$Cayley(S,I)$. It is also clear that once we have a $0$ in $f(w)$
all elements beyond it will be $0$. Let
$w=[(a_1,b_1),(a_2.b_2),\ldots,(a_k,b_k)]$, and $\varphi_s$ a
generator.
\begin{align*}
&\varphi_s(w) = \\
&[s(a_1,b_1),s(a_1,b_1)(a_2,b_2),\ldots,s(a_1,b_1) \cdots
(a_k,b_k)]= \\ &[s(a_1,b_1),s(a_1,b_2),\ldots,s(a_1,b_k)]
\end{align*}
\item
We cannot have a $0$ in the first $k$ entries of $w$ because then we
would also have a $0$ in the first $k$ entries of $f(w)$. We see
that applying $\varphi_s$ to $w$ doesn't change $b_i$ for $i \leq
k$.
\item
The last statement is immediate.
\end{enumerate}
\end{proof}

The following lemma will be needed to cover the last cases for the
main theorem of this section.
\begin{lemma} \label{end cases}
Let $\varphi_s \in Cayley(S,I)$ be a generator and
$w=[\ldots,(a_i,b_i),(a_{i+1},b_{i+1}),\ldots]\in I^*$. Furthermore
assume that the first $i-1$ entries of $w$ are $\neq 0$.
\begin{enumerate}
\item
Assume $\varphi_s(w)$ is non zero in the $i$ position. Then
$\varphi_s(w)$ has a zero in the $i+1$ position if and only
$C(b_i,a_{i+1})=0$ where $C$ is the structure matrix of $I\setminus
\{0\}$.
\item
Let $f=\varphi_{s_n}\cdots\varphi_{s_2}\varphi_{s_1}$ and let
$w=[(a,b),\ldots]$. Then $f(w)=[0,0,\ldots,0]$ if and only if
$s_n\cdots s_2s_1(a,b)=0$.
\end{enumerate}
\end{lemma}

\begin{proof}
A direct calculation can verify both facts. We notice that for the
second statement there are two possibilities for $f$ when $s_n\cdots
s_2s_1(a,b)=0$. Either $s_n \cdots s_2s_1=0$ in which case $f$ is
the zero function. If not, $(s_n\cdots s_2s_1)(a,b)=0$ in which case
$f$ is not the zero function, but $f$ acts as the zero function on
any string of the form $[(a,b'),\ldots]$ for any $b' \in B$.
\end{proof}

\begin{prop}
Let $S,I$ be as before. Then $Cayley(S,I) \prec mem(S)$.
\end{prop}

The proof is motivated by the following idea: Let $v\in I^*$, and
let $f=\varphi_{s_n}\cdots \varphi_{s_1} \in Cayley(S,I)$. The
output $f(v)$, is of the form \\
$[(a,b_1),(a,b_2),\ldots,(a,b_k),0,0,\ldots,0]$. Applying $f$
doesn't change the $b_i$'s of $v$. The action of $f$ on $v$ is
determined by the following two questions: 1. what is $a$? 2. what
is $k$? (i.e. from what point does $f(v) $ consist of zeros?).

If $v=[(a_1,b_1),\ldots]$ then we answer the first question by
$(a,b_1)=s_n\cdots s_2s_1(a_1,b_1)$, i.e. the first question is
answered by the first component of an element of $mem(S)$.

For the second we will observe that the appearance of a $0$ will
depend on the set-component of $mem(S)$. Each time we apply a
$\varphi_{s_i}$ to $v$ we get a string of the form $[s_i\cdots
s_2s_1(a_1,b_1),\ldots]$, so the set-component determines the set of
$a$'s we will see along the way. Part 2 of Lemma \ref{end cases}
tells us that if
$v'=\varphi_{s_i}\cdots\varphi_{s_1}(v)=[(a,b_1),\ldots]$, then a
$0$ will be created in $v'$ when we find a $b_i$ such that
$C(b_i,a)=0$.

\begin{proof}
Denote by $S^+$ the free semigroup over $S$, with $S$ considered as
a set. Denote by $\Gamma:S^+ \to Cayley(S,I)$ the canonical
morphism. $$\Gamma:([s_n,\ldots,s_2,s_1]) \mapsto
\varphi_{s_n}\cdots\varphi_{s_2}\varphi_{s_1}$$

Let $\Phi:S^+ \to mem(S)$ be defined by $$\Phi:([s_1]) \mapsto
(s_1,\emptyset)$$ By induction on $n$ it is easy to see that
$$\Phi:([s_n,\cdots,s_2,s_1]) \mapsto (s_n\cdots s_2s_1,\{s_{n-1}\cdots
s_2s_1,\ldots,s_2s_1,s_1\})$$

We will denote by $\sim_\Gamma$ and $\sim_\Phi$ the congruences on
$S^+$ corresponding to $\Gamma$ and $\Phi$.

We show that $v_1\sim_\Phi v_2$ implies $v_1 \sim_\Gamma v_2$.
Let $v_1,v_2 \in S^+$ and let $f=\Gamma(v_1),g=\Gamma(v_2)$. Let
$v_1=[s_n,\ldots,s_2,s_1]$, and $v_2=[t_m,\ldots,t_2,t_1]$. We want
to show that $\Phi(v_1)=\Phi(v_2)$ implies
$\Gamma(v_1)=\Gamma(v_2)$. This is equivalent to $f=g$ in
$Cayley(S,I)$ if $(s_n\cdots s_1,\{s_{n-1}\cdots
s_1,\ldots,s_1\})=(t_m\cdots t_1,\{t_{m-1}\cdots t_1,\ldots,t_1\})$.
If $f$ is the zero function then $s_n\cdots s_1(a,b)=0$ for every
$(a,b)\in A \times B$. Thus, $\Phi(v_1)=(s_n\cdots s_1,\alpha)$ for some $\alpha \in P(S)$,
and since $\Phi(v_1)=\Phi(v_2)$, $\Phi(v_2)=(s_n\cdots s_1,\alpha)$.
In other words, $s_n\cdots s_2s_1=t_m \cdots t_2 t_1$ and we see
that $g$ is also the zero function on $I^*$.

Throughout the rest of this proof, we will assume $f \neq \varphi_0
\neq g$. Write $\Phi(v_1)=\Phi(v_2)=(\sigma,\alpha)$,$\sigma \in
S$,$\alpha \in P(S)$, with $\sigma(a,b) \neq 0$ for some $(a,b)$.
Finally, let $w \in I^*$.

Let $k_1$ denote the position of the last nonzero entry in $f(w)$
(and $k_1=0$ if $f(w)=[0,0,\ldots,0]$) and $k_2$ the position of the
last nonzero entry in $g(w)$. By Lemma \ref{almost
simple}, and since $\Phi(v_1)=\Phi(v_2)$ (in particular $s_n\cdots
s_2s_1 = t_m \cdots t_2t_1$) it follows that $f(w)$ equals $g(w)$ in
the first $min(k_1,k_2)$ entries.

It remains to show that $k_1=k_2$.

First, assume $k_1=0$. Let $w=[(a_1,b_1),\ldots]$. This means that
$f(w)=[0,0,\ldots,0]$. Since $f \neq \varphi_0$ this means that
$C(s_n\cdots s_2s_1,a_1)=0$ i.e. $C(\sigma,a_1)=0$ and so
$g(w)=[0,0,\ldots,0]$ and $k_2=0$, so $k_1=k_2$.

We may assume now that $k_1,k_2 > 0$. Without loss of generality
$k_1 \leq k_2$.

If $k_1=|w|$, then all entries of $f(w)$ are nonzero and since
$k_1\leq k_2$ we are done in this case. So we may assume that $k_1 <
|w|$.

If $w$ has a zero in the $k_1+1$ place then it is immediate that
$g(w)$ has a zero in the $k_1+1$ place, so in this case $k_2 \leq
k_1$. We may therefore assume that the first $k_1+1$ entries in $w$
are not zero, and \\
$w=[(a_1,b_1),(a_2,b_2),\ldots,(a_{k_1},b_{k_1}),(a_{k_1+1},b_{k_1+1}),\ldots]$.


We know that $f(w)$ has a zero in the $k_1+1$ place,and
$f=\varphi_{s_n}\cdots\varphi_{s_2}\varphi_{s_1}$. Let $r$ be such
that $\varphi_{s_r}\cdots\varphi_{s_2}\varphi_{s_1}(w)$ has a zero
in the $k_1+1$ place but
$\varphi_{s_{r-1}}\cdots\varphi_{s_2}\varphi_{s_1}(w)$ does not.

If $r=1$ (i.e. if $\varphi_{s_1}(w)$ has a zero in the $k_1+1$
position) then part 1 of Lemma \ref{end cases} shows that $\varphi_{t_1}(w)$
has a zero in the $k_1+1$ position, and so $g(w)$ has a zero in the
$k_1+1$ position. Thus $k_2 \leq k_1$ and the claim is proved.

Assume $r>1$. Let
$w'=\varphi_{s_{r-1}}\cdots\varphi_{s_1}(w)=[(c,b_1),\ldots,(c,b_{k_1}),(c,b_{k_1+1}),\ldots]$,
where $c$ is determined by $s_{r-1}\cdots
s_2s_1(a_1,b_1)=(s_{r-1}\cdots s_2s_1a_1,b_1)=(c,b_1)$. Applying
$\varphi_{s_r}$ to $w'$ creates a zero in the $k_1+1$ place and
leaves the $k_1$ place nonzero. We calculate
$\varphi_{s_r}(w')=[(s_rc,b_1),\ldots,(s_rc,b_{k_1}),(s_rc,b_{k_1})(c,b_{k_1+1}),\ldots]$.
Since $(s_rc,b_{k_1}) \neq 0$ and $(s_rc,b_{k_1})(c,b_{k_1+1})=0$ we
see that the structure matrix of $I$ has $C(b_{k_1},c)=0$. 
Thus, $C(b_{k_1},s_{r-1}\cdots s_2s_1)=0$.

We observe that for $\Phi(v_2)=(\sigma,\alpha)$, $\alpha$ describes
the set of elements of $S$ that $w$ is multiplied by, as we apply
$\varphi_{s_i}$, $1\leq i \leq n$. Since $\Phi(v_1)=\Phi(v_2)$, and
since $s_{r-1}\cdots s_2s_1 \in \alpha$,we see that $s_{r-1}\cdots
s_1$ belongs to the set-component of $\Phi(v_2)$. This implies
that 
$g$ can also be decomposed as $g=g_2\varphi_tg_1$ for some $t \in S$ with $g_1 \in
Cayley(S,I)$, $g_2 \in Cayley(S,I) \cup \{id\}$, and
$g_1([(a_1,b_1)]) = [s_{r-1}\cdots s_2s_1(a_1,b_1)]$.

We write $g_1(w)=[(d,b_1),\ldots,(d,b_{k_1}),(d,b_{k_1+1}),\ldots]$
(We assume that the $k_1$ and $k_1+1$ places are not zero or else we
will have $k_2 \leq k_1$ and the proof is over). Since
$[(d,b_1)]=g_1[(a_1,b_1)]=[(s_{r-1}\cdots
s_2s_1a_1,b_1)]=f_r[(a_1,b_1)]=[(c,b_1)]$ we see that $c=d$. We
rewrite, $g_1(w)=[(c,b_1),\ldots,(c,b_{k_1}),(c,b_{k_1+1}),\ldots]$.

Calculating
$\varphi_tg_1(w)=[t(c,b_1),\ldots,t(c,b_1)\cdots(c,b_{k_1})(c,b_{k_1+1})]$,
we see that the $k_1+1$ position has a right factor
$(c,b_{k_1})\cdot(c,b_{k_1+1})$. We noticed before that the
structure matrix has $C(b_{k_1},c)=0$, so $\varphi_tg_1(w)$ has a
zero in the $k_1+1$ position, and so does $g(w)$. We have showed
that $k_2 \leq k_1$.

We now have $k_1=k_2$. Thus we have showed that $f(w)=g(w)$ and thus
$f=g$. Thus, $\Phi(v_1)=\Phi(v_2)$ implies
$\Gamma(v_1)=f=g=\Gamma(v_2)$.


Since we have proved that $\sim_\Phi \subset \sim_\Gamma$, there is
an induced surjective morphism from $S^+/\sim_\Phi$ onto
$Cayley(S,I)\S^+/\sim_\Gamma$. Therefore, $Cayley(S,I)$ divides
$S^+/\sim\Phi$ which is a subsemigroup of $mem(S)$.

\end{proof}

\begin{thm} \label{0-simple is finite}
Let $S$ be a finite aperiodic semigroup, and $I$ a $0$-minimal ideal. Then
$Cayley(S,I)$ is finite and aperiodic.
\end{thm}

This was proved for the regular case. For a nonregular
$0$-minimal ideal, the theorem is trivial. This is
because such a $\mathcal J$ class is null, so for any $\varphi_s \in
Cayley(S,I)$ and $w=[w_1,w_2,\ldots] \in I^*$ we have
$\varphi_s(w)=[sw_1,sw_1w_2=0,0,\ldots,0]$, so the action of
$\varphi_{s_n}\cdots\varphi_{s_1}$ on $w$ is determined by
$s_n\cdots s_1$, with
$\varphi_{s_n}\cdots\varphi_{s_1}(w)=[s_n\cdots
s_2s_1w_1,0,\ldots,0]$ and $Cayley(S,I)$ is the quotient
of $S$ obtained by identifying two elements if they act the same on
the left of $I$.

For future purposes we extend the previous result.
\begin{defn}
Let $S$ be a finite aperiodic semigroup, and let $I$ be a
$0$-minimal ideal of $S$. Let $f \in Cayley(S,I)$, and
$w=[v_1,\ldots,v_l]\in I^*$. We say that a zero \emph{is created} in
the $i$th position of $w$ by $f$, if there exists a decomposition of
$f$, $f=\varphi_{s_n}\cdots\varphi_{s_2}\varphi_{s_1}$ such that
\begin{enumerate}
\item
$\varphi_{s_{k-1}}\cdots\varphi_{s_2}\varphi_{s_1}(w)$ does not have
a zero in the $i$ position.
\item
$\varphi_{s_k}\cdots\varphi_{s_2}\varphi_{s_1}(w)$ has a zero in the
$i$ position but is nonzero in the $i-1$ position.
\end{enumerate}
If we need to be more specific we will say that the zero is created
by the $k$th generator of the above decomposition of $f$.
\end{defn}

\begin{defn}
Let $S$ be a finite aperiodic semigroup, and $I$ a $0$-minimal ideal
of $S$, such that $I \cong (A\times B) \wzero$. Let
$f=\varphi_{s_n}\cdots\varphi_{s_2}\varphi_{s_1} \in Cayley(S,I)$
and $a\in A$, such that $s_n\cdots s_2s_1(a,b)\neq 0$ for any (and
therefore all) $b \in B$.
\begin{enumerate}
\item
For $1\leq i \leq n+1$ let $\pi_a(i)=s_{i-1}\cdots s_2s_1a$ (and
$\pi_a(1)=a$). We recall the notation $sa$ from Notation \ref{nota sa}.
\item
$new_a(s_n,\cdots,s_1)=\{1\leq i \leq n|\pi_a(i) \neq \pi_a(j) \forall j<i\}$.
\end{enumerate}
\end{defn}

Notice that $|new_a(s_n,\ldots,s_1)| \leq |A|$, for any
$s_n,\ldots,s_1 \in S$.

\begin{lemma}
Let $f=\varphi_{s_n}\cdots\varphi_{s_2}\varphi_{s_1}$. Let $w\in
Cayley(S)(I^*)$ (i.e. $w$ is the output of $Cayley(S,I)$ acting on
some string in $I^*$) such that $w$ is not a string of $0$s. Let
$w=[(a_1,b_1),\ldots]$ and let $a=a_1$. If the $k$th generator of
$f$ creates a new zero, then $k \in new_a(s_n,\ldots,s_2,s_1)$.
\end{lemma}

\begin{proof}
Suppose the $k$th generator creates a zero in the $i$th entry.
Without loss of generality
$f=\varphi_{s_k}\cdots\varphi_{s_2}\varphi_{s_1}$. Let
$f'=\varphi_{s_{k-1}}\cdots\varphi_{s_2}\varphi_{s_1}$. We then have
$f'(w)=[(a',b_1),\ldots,(a',b_{i-1}),(a',b_i),\ldots]$, where
$a'=s_{k-1}\cdots s_2s_1a=\pi_a(k)$. This is true even if $k=1$ and
$f'$ is the identity function, since we assume $w \in
Cayley(S)(I^*)$. Since the zero is created in the $i$th entry we
have
\\$f(w) = \varphi_{s_k}(f'(w)) =
[(s_ka',b_1),\ldots,(s_ka',b_{i-1}),(s_ka',b_{i-1})(a',b_i)=0,\ldots]$.
This shows that
$C(b_{i-1},a')=0$. 
If $k=1$ then it is immediate that $k \in new_a(s_n,\ldots,s_1)$.
Suppose that $1<k \notin new_a(s_n,\ldots,s_1)$. Then for some $l <
k$ we have $s_{l-1}\cdots s_2s_1a=a'$ and by a similar calculation
the zero will be created by the $l$th generator and not the $k$th
generator.
\end{proof}

\begin{cor} \label{simple finiteness}
Let $S$ be an aperiodic semigroup, and let $I$ be a $0$-minimal
ideal of $S$. There exists an $N\leq |A|^2+1 \in \mathbb N$ such
that
every $f=\varphi_{s_n}\cdots \varphi_{s_2}\varphi_{s_1}
\in Cayley(S,I)$ can be written as
$f=g_{k+1}\varphi_{s_{i_k}}g_k\cdots g_2 \varphi_{s_{i_1}}g_1$ for
some $k \leq N$, with $g_l \in Cayley(S,I) \cup \{id\}$ and such
that all zeros are created by the $\varphi_{s_{i_j}}$, $1 \leq j
\leq k$ for every $w \in I^*$,
\end{cor}

\begin{proof}
In order to use the previous lemma $w$ needs to be in $Cayley(S)(I^*)$.
Instead of $f=\varphi_{s_n}\cdots\varphi_{s_2}\varphi_{s_1}$ acting
on $w$ we may consider $f=\varphi_{s_n}\cdots\varphi_{s_2}$ acting
on $\varphi_{s_1}(w)$.

For every $a \in A$ consider the set $new_a(s_n,\ldots,s_1)$, and
let\\ $X=\cup_{a \in A}new_a(s_n,\ldots,s_1)$. By the previous
lemma, all the zeros must be created by the $r$th generator for some $r \in X$. In fact,
every $a \in A$ contributes at most $|A|$ elements to $X$. Thus, for any word $w$, the zeroes will be created by an $r$th generator for some $r \in X$. Thus, in the product $f=\varphi_{s_n}\cdots\varphi_{s_2}\varphi_{s_1}$ we can single out $|X|$ generators which are the only ones that can create a zero. We also have
$|X| \leq |A|^2+1$. The $+1$ in this formula comes from the
beginning of the proof.
\end{proof}

\section{Finiteness}

Before the next lemma, recall the definition of $p:S^+ \to S$ given
by $p([a_1,a_2,\ldots,a_l])=a_l$ and in particular, $p([a])=a$.

\begin{lemma} \label{finiteness_technical_lemma}
Let $S$ be a finite aperiodic semigroup, and $A \times B$ a regular
$\mathcal J$ class of $S$. Let $[v_1,v_2,\ldots,v_k] = v \in (A
\times B)^+$. Let $f \in Cayley(S)$. Write
$f=\varphi_{s_n}\cdots\varphi_{s_1}$ and $v_i=(a_i,b_i)$.
Furthermore assume that
 $f(v) \in (A \times B)^+$. Then
$$f_v=\varphi_{(s_n\cdots
s_2s_1a_1,b_k)}\cdots\varphi_{(
s_2s_1a_1,b_k)}\varphi_{(s_1a_1,b_k)}$$.
\end{lemma}

\begin{proof}
Since $f(v) \in (A \times B)^*$ we have $v_1v_2\cdots v_k
\in A \times B$, 
so $v_1v_2\cdots v_i=(a_1,b_i)$ for $1\leq i \leq k$.
$\varphi_{s_1}(v)=[s_1v_1,s_1v_1v_2,\ldots,s_1v_1\cdots
v_k]=[(s_1a_1,b_1),\ldots,(s_1a_1,b_k)]$. Proceeding from here (by
inducting on $i$) we have
$$\varphi_{s_i}\cdots\varphi_{s_1}(v)=[(s_i\cdots s_2s_1
a_1,b_1),\ldots,(s_i\cdots s_2s_1 a_1,b_k)]$$

We now recall the formula for $f_v$ from Theorem \ref{fractalness}.
We have $f_v=\varphi_{(p\varphi_{s_n}\cdots\varphi_{s_1}(v))} \ldots
\varphi_{(p\varphi_{s_2}\varphi_{s_1}(v))}
\varphi_{(p\varphi_{s_1}(v))}$. 
Since $p\varphi_{s_i}\cdots\varphi_{s_1}(v)=(s_i \cdots
s_2s_1a_1,b_k)$, we have our result.
\end{proof}


\begin{rem}
With the previous definitions, assume that $f(v) \in (A \times
B)^+$. Then $\varphi_1(v) \in (A\times B)^+$.
\end{rem}

\begin{proof}
We write $f=\varphi_{s_n}\cdots\varphi_{s_2}\varphi_{s_1}$.
We observe that every entry in the word
$\varphi_1([v_1,v_2,\ldots,v_k])=[v_1,v_1v_2,\ldots,v_1v_2\cdots
v_k]$, is $\mathcal L$ above the corresponding entry of $f(v)$. This follows from the fact that $f(v)$ can be obtained from $\varphi_1(v)$ by multiplying each element on the left by $s_1$, and then by applying $\varphi_{s_n}\cdots\varphi_{s_2}$ to the resulting word, which again is just left multiplication.
Since $v \in (A\times B)^+$ and $f(v)\in (A\times B)^+$ we have that
$\varphi_1(v) \in (A \times B)^+$.
\end{proof}

\begin{cor} \label{action_on_suffix_depends_on_j}
Let $S$ be a finite aperiodic semigroup, $A \times B$ a $\mathcal J$
class of $S$. Let $[v_1,v_2,\ldots,v_k] = v \in (A \times B)^+$ and
let $f \in Cayley(S)$. Write $f=\varphi_{s_n}\cdots\varphi_{s_1}$
and $v_i=(a_i,b_i)$. Assume that $f(v) \in (A \times B)^+$.
Furthermore, let $j=p\varphi_1(v)=(a_1,b_k)$. Then $f_v=f_{[j]}$.
\end{cor}

\begin{proof}
We first show that $f([j]) \in (A \times B)^+$. We calculate
$f([j])=[s_n\cdots s_2s_1(a_1,b_k)]$. Since $f(v) \in (A \times
B)^+$, and since the first letter of $f(v)$ is $s_n\cdots
s_2s_1(a_1,b_1)$ it follows that $s_n\cdots s_2s_1(a_1,b_k) \in A
\times B$, so $f([j]) \in (A \times B)^+$.

Now, $f_v=\varphi_{(s_n\cdots s_2s_1(a_1,b_k))}\cdots\varphi_{(
s_2s_1(a_1,b_k))}\varphi_{(s_1(a_1,b_k))}$. We also have
$f_{[j]}=\varphi_{(s_n\cdots s_2s_1(a_1,b_k))}\cdots\varphi_{(
s_2s_1(a_1,b_k))}\varphi_{(s_1(a_1,b_k))}$.

\end{proof}

\begin{defn}
Let $S$ be a finite semigroup. Let $T$ be an ideal of $S$ and $J$ a
$\mathcal J$ class of $S$ not in $T$. We say that $J$ is
\emph{directly above} $T$ if $J \cup T$ is an ideal of $S$.
Equivalently, $J$ is a $0$-minimal $\mathcal J$ class of $S / T$.
\end{defn}

Let $S$ be a semigroup with an ideal $T$ and a $\mathcal J$ class
$J$ directly above $T$. Let $w \in (T \cup J)^*$. Since every letter
in $w$ is in $J$ or in $T$, $w$ can be uniquely decomposed as
$w=w_J$ or $w=w_Jw_T'w'$ such that $w_J \in J^*$, $w_T' \in T$ and
$w' \in (J \cup T)^*$. In the case that $w$ is of the form
$w=f(w_0)$ for some $f \in Cayley(S,T \cup J)$ and $w_0 \in (T \cup
J)^*$ we can decompose $w=w_Jw_T$ with $w_J \in J^*$ and $w_T \in
T^*$. In both cases, we call $w_J$ the prefix of $w$ in $J$.

\begin{defn}
Let $S$ be a finite aperiodic semigroup. Let $T$ be an ideal of $S$,
and let $J$ be a $\mathcal J$ class directly above $T$. For an $f\in
Cayley(S,T\cup J)$, the set of $f$-$J$-stable words is the set of
words $st(f,J) \subset (T \cup J)^*$ such that the prefix of $w$ in
$J$ has the same length as the prefix of $f(w)$ in $J$.
\end{defn}

We observe that in general, the length of the prefix of $w$ in $J$
is larger or equal to the length of the prefix of $f(w)$ in $J$.

Consider the following relation on $Cayley(S,T \cup J)$. We write $f
\sim g$ if and only if $st(f,J)=st(g,J)$ and the restrictions of $f$
and $g$ to this set is equal.

\begin{lemma}
Let $S,T,J,\sim$ be as before. $\sim$ is a congruence on
\\$Cayley(S,T \cup J)$.
\end{lemma}

\begin{proof}
It is clear that $\sim$ is an equivalence relation. Suppose that $f
\sim f'$ and $g \sim g'$. Let $w\in st(fg,J)$. In particular this
means that $w\in st(g,J)$ and that $g(w) \in st(f,J)$. This implies
$g(w)=g'(w)$ and that $f(g(w))=f'(g(w))=f'(g'(w))$. Thus $w \in
st(f'g',J)$ and $fg(w)=f'(g(w))=f'g'(w)$. The dual argument
completes the proof.
\end{proof}

\begin{nota}
We write $Cayley_{st-J}(S,T \cup J)$ for the quotient of
\\$Cayley(S,T \cup J)$ by $\sim$.
\end{nota}

We point out, that for $[f]_\sim\in Cayley_{st-J}(S,T \cup J)$ and
$w \in st(f,J)$, $[f]_\sim(w)$ is always well defined, i.e. every $g \sim f$
can act on $w$ and all these $g$'s give the same $g(w)$. We will
refer to $[f]_\sim(w)$ as $\tilde f(w)$, whenever $[f]_\sim(w)$ is defined.
In general we will write $\tilde f$ for $[f]_{\sim}$.

Until now, the notation $Cayley(S,T)$ was used when $T$ was an ideal
of $S$. It can actually be used where $T$ is any semigroup which $S$
acts on from the left, by defining
$\varphi_s([t_1,t_2,\ldots,t_k])=[st_1,st_1t_2,\ldots,st_1t_2\cdots
t_k]$. However, in general, $Cayley(S,T)$ is no longer a divisor of
$Cayley(S)$.

Let $S$ be an aperiodic semigroup, $T$ an ideal of $S$ and $J$ a
$\mathcal J$ class directly above $T$. Let $\alpha:(T \cup J) \to
J^{tr} \cong (T\cup J)/T$ be the
ideal quotient. 
$\varphi_s$ can act on strings in $J^{tr}$ in the following way: Let
$w=[j_1,\ldots, j_n] \in (J^{tr})^*$. We have
$\varphi_s([j_1,j_2,\ldots,j_n])=[\overline{sj_1},\overline{sj_1j_2},\ldots,\overline{sj_1j_2\ldots
j_n}]$, with $\overline s=\alpha(s)$, i.e. $\overline s=s$ if $s \in
J$ and $\overline s=0$ otherwise.

\begin{nota}\label{S acting on J tr}
We write $Cayley(S,J^{tr})$ for the semigroup of transformations
$(J^{tr})^* \to (J^{tr})^*$ generated by $\{\varphi_s\}_{s \in S}$.
\end{nota}

We define a morphism $Cayley(S,T \cup J) \to
Cayley(S,J^{tr})$. Write $f^{tr}$ for the image of $f$. $f^{tr}$'s action
on $w \in (J^{tr})^*$ can be described as follows. We first replace
any $0$ in $w$ with any element of $T$. Next, apply $f$ to this new
word. Finally, replace all elements of $T$ with $0$ (the zero of
$J^{tr}$).

Let $f \in Cayley(S,J \cup T)$ and $w \in (J \cup T)^*$. Let
$f^{tr}$ be as defined above and let $w^{tr}$ denote the word
obtained from $w$ be replacing elements of $T$ with $0$. We easily
see that $w \in st(f,J)$ if and only if $w^{tr} \in st(f^{tr},J)$
($J$ is a $\mathcal J$ class of $J^{tr}$, so $st(f^{tr},J)$ makes
sense).

We will use the above notations throughout the next proofs. We
summarize.

\begin{defn}
Let $S$ be an aperiodic finite semigroup. Let $T$ be an ideal and
let $J$ be a $\mathcal J$ class directly above $T$. We have defined:
\begin{enumerate}
\item
For $w \in (T \cup J)^*$, the word $w^{tr} \in (J^{tr})^*$ is
obtained from $w$ by replacing elements of $T$ with $0$.
\item
Let $f \in Cayley(S,T\cup J)$. The function $f^{tr} \in
Cayley(S,J^{tr})$ is given by $f^{tr}(w)=(f(w'))^{tr}$, where $w'$
is any word in $(J \cup T)^*$ satisfying $(w')^{tr}=w$. It is
obvious that this is well defined.
\item
Let $f \in Cayley(S,T \cup J)$. We define $st(f^{tr},J)$ to
be the set of $f^{tr}$-$J$-stable words in $(J^{tr})^*$. For $w \in
(J \cup T)^*$, we have $w \in st(f,J)$ if and only if $w^{tr} \in
st(f^{tr},J)$.
\end{enumerate}
\end{defn}

We see that the mapping $f \mapsto f^{tr}$ is a morphism $Cayley(S,T
\cup J) \to Cayley(S,J^{tr})$, i.e. $f^{tr}g^{tr}=(fg)^{tr}$. We
will now show that it induces a morphism $Cayley_{st-J}(S,T \cup J)
\to Cayley_{st-J}(S,J^{tr})$ .

\begin{defn}
Let $\tilde f\in Cayley_{st-J}(S,T \cup J)$ and let $w \in (J^{tr})^*$. Write $f^{tr}(w)$ for the word
$(f(w))^{tr}$ where $f$ is an element of $Cayley(S,T \cup J)$ such that $f$ maps to $\tilde f$,
and the action of $f$ on $J^{tr}$ is understood as in Notation \ref{S acting on J tr}. It is clear
that this is well defined.
Thus we have defined a mapping $Cayley_{st-J}(S,T \cup J) \to Cayley_{st-J}(S,J^{tr})$.
\end{defn}

\begin{lemma}
Let $S$ be an aperiodic semigroup, $T$ an ideal of $S$ and $J$ a
$\mathcal J$ class directly above $T$. Let $\tilde f,\tilde g\in Cayley_{st-J}(S,T
\cup J)$, and let $f,g \in Cayley(S,T\cup J)$ be preimages of $\tilde f,\tilde g$. Then $(fg)^{tr}=f^{tr}g^{tr}$.
\end{lemma}

\begin{proof}
This follows from the fact that the mapping $f \mapsto f^{tr}$ is a morphism, and that $\tilde f
\mapsto f^{tr}$ is well defined. Furthermore, we see that this does not depend on the choice of
$f$ and $g$.
\end{proof}

We wish to extend all result we had previously obtained for
$Cayley(S,I)$ where $I$ was a $0$-minimal ideal of $S$. When we consider
$Cayley(S,J^{tr})$, $J^{tr}$ is a $0$-simple semigroup, however it's not
an ideal of $S$.

\begin{thm} \label{0-minimal is like J class}
Let $S$ be an aperiodic semigroup, let $J$ be a $\mathcal J$ class
of $S$, and let $T=S^1JS^1 \setminus J$. Denote by $\sim_T$ the
ideal congruence associated with $T$. Then $(J\cup T)/\sim_T$ is a
$0$-minimal ideal of $S/\sim_T$ and $Cayley(S,J^{tr}) \cong
Cayley(S/\sim_T,(J \cup T)/\sim_T)$. Furthermore, since $\sim_T$ is
$1-1$ on $J$ we have $J/\sim_T=J$.
\end{thm}

\begin{proof}
It is obvious that $(J\cup T)/\sim_T$ is a $0$-minimal ideal of
$S/\sim_T$. To show the isomorphism, consider $\varphi_s \in
Cayley(S,J^{tr})$. If $s \in T$, then the action of $\varphi_s$ on
$(J^{tr})^*$ is seen as sending any string to a string of zeros, so
in this case $\varphi_s$ corresponds to $\varphi_0$ where $0$ is the
$0$ element of $S/\sim_T$ and it sends any string in $((J\cup
T)/\sim_T)^*$ to a string of zeros.

If $s \notin T$ then the action of $\varphi_s$ on $(J^{tr})^*$ is
the same as in \\$Cayley(S/\sim_T,(J \cup T)/\sim_T)$.

We can extend this from $\{\varphi_s\}$, the generators of
$Cayley(S,J^{tr})$ to the whole semigroup, and obtain an isomorphism
$Cayley(S,J^{tr}) \cong Cayley(S/\sim_T,(J \cup T)/\sim_T)$.
\end{proof}

Let $Cayley(S,T)^{J^{tr}}$ denote the semigroup of all functions
$J^{tr} \to Cayley(S,T)$ with pointwise multiplication. We now
define a right action of $Cayley_{st-J}(S,J^{tr})$ on
$Cayley(S,T)^{J^{tr}}$. Let $\tilde f \in Cayley_{st-J}(S,J^{tr})$ and let
$H:J^{tr} \to Cayley(S,T)$. Let $j \in J$. If $[j] \in st(\tilde f,J)$ then
we let $H^{\tilde f}(j)=H(\tilde f(j))$. Otherwise we let $H^{\tilde f}(j)=\varphi_0$ (Recall
that $\varphi_0$, the function sending any string to a string of
zeros is the zero element of $Cayley(S,T)$).

Notice the slight abuse of notation for $\tilde f(j)$ where we identify
elements of the semigroup $S$ with strings of length 1. Formally,
instead of $\tilde f(j)$ we should have written $p(\tilde f([j]))$, where $p$ is
the projection $p([v_1])=v_1$.

To show that this is indeed an action, let $\tilde f,\tilde g \in
Cayley_{st-J}(S,J^{tr})$. We first need to show that
$H^{\tilde f\tilde g}(j)=(H^{\tilde f})^{\tilde g}(j)$ for every $j \in J$. Consider the following
cases:
\begin{enumerate}
\item
If $j=0$ then both sides are $H(0)$. $[0]$ is a $J$-stable word for
every function in $Cayley_{st-J}(S,T \cup J)$.
\item
If $j \in J$ and $[j] \in st(\tilde f\tilde g,J)$ then we also have $[j] \in
st(\tilde g,J)$, $j'=\tilde g(j)$ and $[\tilde f(j')] \in st(\tilde g,J)$. Now both sides of the
equation equal $H(\tilde f(j'))$.
\item
If $j \in J$ and $[j] \notin st(\tilde f\tilde g,J)$. This implies that either
$[j] \notin st(\tilde g,J)$ or $[\tilde g(j)] \notin st(\tilde f,J)$. In both cases this
gives $\varphi_0$ on both sides of the equation.
\end{enumerate}

It is clear that $(FG)^{\tilde h}=F^{\tilde h}G^{\tilde h}$ for $F,G:J^{tr} \to Cayley(S,T)$,
and $\tilde h \in Cayley_{st-J}(S,J^{tr})$.

We may now define the semidirect product $Cayley_{st-J}(S,J^{tr})
\rtimes Cayley(S,T)^{J^{tr}}$.

Now for the main theorem of this section.
\begin{thm}
Let $S$ be a finite semigroup. Let $T$ be an ideal of $S$, and let
$J$ be a $\mathcal J$ class of $S$ directly above $T$.
$Cayley_{st-J}(S,T\cup J)$ is isomorphic to a subsemigroup of $
Cayley_{st-J}(S,J^{tr}) \rtimes (Cayley(S,T))^{J^{tr}}$.
\end{thm}

Before proceeding with the proof we would like to show the reasoning
behind this theorem. The action of $\tilde f$ on a stable word $w$ can be
thought of as two actions. The first is modifying the prefix of $w$
in $J$ to some other prefix in $J$. This can be thought of as
$\tilde f\mapsto f^{tr}$. The second part is the action of $\tilde f$ is on the
remainder, the part in $T^*$. In general this action depends on the
prefix of $w$ in $J$. However, Corollary
\ref{action_on_suffix_depends_on_j} says that this action only
depends on an element in $J$.

\begin{proof}

Let $\tilde f\in Cayley_{st-J}(S,T \cup J)$ and let $f$ be any element in $Cayley(S,J \cup T)$
 that maps onto $\tilde f$. Denote by $\hat f$ the function
$\hat f:J^{tr} \to Cayley(S,T)$ defined by
$$(\hat f(j))(w)=\begin{cases}f(w) & j=0 \\ f_{[j]}(w) & [j]\in st(f,J),j\in J\\\varphi_0(w) & [j] \notin st(f,j),j \in J\end{cases}$$
restricted to $T^*$ (i.e. $w \in T^*$). This is well defined. The
first case because every word in $T^*$ is $J-$stable. The second
case, since if $f(j) \in J$, then for every $w=[t_1,\ldots,t_k]\in
T^*$ we can express $f_{[j]}(w)$ as the $k$-length suffix of
$f([j,t_1,t_2,\ldots,t_k])$. Furthermore this does not depend on the choice of $f$.


Let $\Pi:Cayley_{st-J}(S,T\cup J) \to Cayley_{st-J}(S,J^{tr})
\rtimes Cayley(S,T)^{J^{tr}}$ be given by $\Pi(\tilde f)\mapsto
(f^{tr},\hat f)$, $\Pi$ is obviously well-defined. To show that it
is a semigroup morphism, consider $\tilde f,\tilde g \in Cayley_{st-J}(S,T \cup
J)$. We have $\Pi(\tilde f)\Pi(\tilde g)=(f^{tr},\hat f)(g^{tr},\hat g)$.
We have seen that $(fg)^{tr}=f^{tr}g^{tr}$. In order to show that $\Pi$ is a
morphism, it remains to show that for every $j \in J^{tr}$ the
equation $(\widehat{fg}(j))(w)=({\hat f}^{(g^{tr})}(j)\hat g(j))(w)$
holds for every $w \in T^*$.

If we choose $j$ as the zero element of $J^{tr}$ the equation is
immediate. Choose a $j \in J$. If $[j] \notin st(g^{tr},J)$ then we
also have $[j]\notin st((fg)^{tr},J)$, so
$\widehat{fg}(j)=\varphi_0={\widehat f}^{(g^{tr})}(j)
\varphi_0={\hat f}^{(g^{tr})}(j)\widehat g(j)$. If $[j]\in
st(g^{tr},J)$, $g^{tr}([j])=[j']$ and $[j']\notin st(f^{tr},J)$,
then as before, we have $[j] \notin st((fg)^{tr},J)$ so
$\widehat{fg}(j)=\varphi_0=\varphi_0 \widehat g(j)={\hat
f}^{(g^{tr})}(j)\widehat g(j)$. Finally, we consider the case $[j]\in
st((fg)^{tr},J)$.

In this case $\widehat{fg}(j)=(fg)_{[j]}$ , $\hat g(j)=g_{[j]}$, and
$(\widehat{f})^{(g^{tr})}(j)=f_{[g(j)]}$, all restricted to $T^*$.
Let $\varphi_{s_n}\cdots\varphi_{s_1}$ be any representative of $f$,
and $\varphi_{t_m}\cdots\varphi_{t_1}$ a representative of $g$. Let
$\bar t=t_m\cdots t_1$. By Lemma \ref{finiteness_technical_lemma}
$\widehat{fg}(j) = (fg)_{[j]} =
(\varphi_{s_n}\cdots\varphi_{s_1}\varphi_{t_m}\cdots\varphi_{t_1})_{[j]}
=\\ \varphi_{(s_n\cdots s_1t_m\cdots
t_1j)}\cdots\varphi_{(s_1t_m\cdots t_1j)}\varphi_{(t_m\cdots
t_1j)}\cdots\varphi_{(t_1j)} = f_{[\bar tj]}g_{[j]} = \hat f(\bar
tj)\hat g(j)$. Since $[j]\in st(g^{tr},J)$, we have
$g^{tr}(j)=t_m\cdots t_1j=\bar tj\in J$. Furthermore, since  $[j]\in
st((fg)^{tr},J)$ we have $g^{tr}([j]) \in st(f^{tr},J)$, i.e. $[\bar
t j] \in st(f^{tr},J)$.
This completes the proof that $\Pi$ is a morphism.

We now show that $\Pi$ is $1-1$. Let $\tilde f,\tilde g\in Cayley_{st-J}(S,T \cup
J)$ and assume that $\Pi(\tilde f)=\Pi(\tilde g)$. As before we will write
$\Pi(\tilde f)=(f^{tr},\hat f)$ and $\Pi(\tilde g)=(g^{tr},\hat g)$. Let $w \in
st(f,J)$. We need to show that $\tilde f(w)=\tilde g(w)$ (this will show that $w
\in st(g,J)$). Furthermore, since $f(w)=\tilde f(w)$ this is the same as showing
that $f(w)=g(w)$. If $w \in T^*$ then this is clear since $f(w)=(\hat
f(0))(w)=(\hat g(0))(w)=g(w)$. If $w \in J^*$, then (recall that $w$
is in $st(f,J)$ so $f(w)$ is in $J^*$) $f(w) = f^{tr}(w) =
g^{tr}(w)$. Since $g^{tr}(w)$ has no zeros (because
$g^{tr}(w)=f^{tr}(w)$), we have that $g^{tr}(w)=g(w)$ . We now
assume that $f(w)$ has a prefix in $J^*$ and a suffix in $T^*$.
Write this as $f(w)=w'=w_1'\circ w_2'$ and write a decomposition for $w$
as $w=w_1\circ w_2$ such that $|w_1|=|w_1'|$ and $|w_2|=|w_2'|$. $w\in
st(f,J)$ implies that $w_1\in J^*$ and $w_2\in T(T \cup J)^*$ (the
first letter of $w_2$ is in $T$). We need to show that $g(w)=w'$. By
definition we have $g(w)=g(w_1)\circ g_{w_1}(w_2)$. Recall that $f(w_1)\in
J^*$, which
implies that $f(w_1)=f^{tr}(w_1)=g^{tr}(w_1)=g(w_1)$, 
It remains to show that $g_{w_1}(w_2)=w_2'=f_{w_1}(w_2)$. Let
$w_1=[(a_1,b_1),\cdots,(a_n,b_n)]$, and let
$j=p\varphi_1(w_1)=(a_1,b_n)$. Corollary
\ref{action_on_suffix_depends_on_j} shows that $f_{w_1}=f_{[j]}$,
and $g_{w_1}=g_{[j]}$. In particular the restriction of this
equation to $T^*$ remains valid. Summing up, we have
$f_{w_1}(w_2)=f_{[j]}(w_2)=\hat f(j)(w_2)=\hat
g(j)(w_2)=g_{[j]}(w_2)=g_{w_1}(w_2)$ which completes the proof.
\end{proof}

\begin{cor} \label{finite stable}
Let $S$ be an aperiodic semigroup, $T$ an ideal of $S$, and $J$ a
$\mathcal J$ class directly above $T$. If $Cayley(S,T)$ is finite
then $Cayley_{st-J}(S,T \cup J)$ is finite.
\end{cor}

\begin{proof}
It is enough to show that $Cayley_{st-J}(S,J^{tr})$ is finite. This
follows from Theorem \ref{0-simple is finite} and Theorem
\ref{0-minimal is like J class}. When applying Theorem
\ref{0-minimal is like J class} we recall that the quotient of a finite
aperiodic semigroup is also aperiodic.


\end{proof}

\begin{thm}
Let $S$ be an aperiodic semigroup, $T$ an ideal of $S$, and $J$ a
$\mathcal J$ class directly above $T$. If $Cayley(S,T)$ is finite
then $Cayley(S,T\cup J)$ is finite.
\end{thm}

\begin{proof}
We will do this by showing that for some constant $N$, every $f\in
Cayley(S,T\cup J)$ can be written as a product of $N$ or less
generators.

Let $f\in Cayley(S,T \cup J)$ be given, and let
$f=\varphi_{s_n}\cdots\varphi_{s_2}\varphi_{s_1}$ be a presentation
of $f$. Since the mapping $f \mapsto f^{tr}$ is a morphism
$Cayley(S,T \cup J) \to Cayley(S,J^{tr})$, we have
$f^{tr}=\varphi_{s_n}^{tr}\cdots\varphi_{s_2}^{tr}\varphi_{s_1}^{tr}$,
an element in $Cayley(S,J^{tr})$.

We observe here that the elements $\{\varphi_s^{tr}\}_{s \in S}$ are
a generating set for $Cayley(S,J^{tr})$.

By Corollary \ref{simple finiteness} there is a constant $N'$
(depending only on $J$) such that we can write $f^{tr} =
g_{k+1}^{tr}\varphi_{s_k}^{tr}g_{k}^{tr}\varphi_{s_{k-1}}^{tr}
\cdots \varphi_{s_2}^{tr}g_{2}^{tr}\varphi_{s_1}^{tr}g_{1}^{tr}$ for
some $k<N'$, where the $g_{i}^{tr}$'s never create new zeros on any
input.

It is obvious that for any word $w\in (T \cup J)^*$, $w$ is
$g_i$-$J$ stable if and only if $w^{tr}$ is $g_i^{tr}$-$J$ stable.
In other words, the $g_i^{tr}$ never create new zeros if and only if
the $g_i$ never map a letter from $J$ to a letter in $T$.

Since the $g_i^{tr}$ never create zeros, we may say that the $g_i^{tr}$ only act
on $g_i^{tr}$-$J$ stable words.
This means that inside the above presentation of $f$, the $g_i$ will
only get to act on $J$-$g_i$ stable words, thus we can identify $g_i$ with $\tilde g_i$,
its image in $Cayley_{st-J}(S,T \cup J)$.

By Corollary  \ref{finite stable} $Cayley_{st-J}(S,T\cup J)$ is
finite, so we may assume there is another constant $l$ such that
every $g_i$ can be written as (or replaced by) a product of $l$ or less generators,
without changing the value of $f(w)$ for any word $w$.

This shows that there is a number $N$ such that every $f \in
Cayley(S,T\cup J)$ can be written as a product of $N$ or less
generators.
\end{proof}

Finally, we have the main theorem of this section.
\begin{prop}
If $S$ is an aperiodic semigroup, $Cayley(S)$ is finite.
\end{prop}

\begin{proof}
Choose an ordering of the $\mathcal J$ classes of $S$ $\{J_i\}_{i
=1}^{n}$, such that for every $1 \leq k \leq n$ we have
$\cup_{i=1}^{k}J_i$ is an ideal of $S$. It is obvious that such an
ordering exists and that $J_1 \cup J_2$ is a $0$-minimal ideal.
We have shown that $Cayley(S,J_1 \cup J_2)$ is finite, and
the last theorem shows that $Cayley(S,\cup_{i=1}^{k}J_i)$ is finite
for every $k$. This completes the proof.
\end{proof}

\section{Aperiodicity}

\begin{lemma}
Let $S$ be a finite semigroup. Let $T$ be an ideal of $S$ such that
$f^n=f^{n+1}$ for every $f \in Cayley(S,T)$ for some $n \in \mathbb
N$. Let $J$ be a $\mathcal J$ class of $S$ directly above $T$, such
that $f^l=f^{l+1}$ for every $f \in Cayley(S,J^{tr})$ for some $l
\in \mathbb N$. Then $f^{n+l}=f^{n+l+1}$ for all $f \in Cayley(S,T
\cup J)$.
\end{lemma}

\begin{proof}
Let $w \in (T \cup J)^*$. Let $f \in Cayley(S,T \cup J)$. Consider
$w_0=(f)^l(w)$. $w_0$ has a prefix $w_J$ in $J^*$ and a suffix $w_T$
in $T^*$. This implies that $(f^{tr})^l(w^{tr})$ has the same prefix
$w_J$ in $J^*$ as $f^l(w)$, and a suffix of zeros. Furthermore,
since $(f^{tr})^l(w^{tr})=(f^{tr})^{l+1}(w^{tr})$, this gives
$f^{tr}(w_J)=w_J$.

We now have
\begin{align*}
&f^{n+l}(w) =\\
&f^n(w_0) =\\
&f^{n}(w_J\circ w_T) = \\
&f^{n-1}(f(w_J)\circ f_{w_J}(w_T)) =\\
&f^{n-1}(w_J\circ f_{w_J}(w_T)) =\\
&\cdots = \\
&(w_J\circ f^{n}_{w_J}(w_T)) = \\
&(w_J\circ f^{n+1}_{w_J}(w_T)) = \\
&f^{n+1}(w_J\circ w_T) =\\
&f^{n+1}(w_0) =\\
&f^{n+l+1}(w).
\end{align*}
\end{proof}

\begin{cor}
Let $S$ be a finite aperiodic semigroup.Then $\exists n \in
\mathbb N$ such that 
$f^n=f^{n+1}$ for all $f \in Cayley(S)$.
\end{cor}

\begin{proof}
As before, let $\{J_i\}_{i=1}^{k}$ be an enumeration of the
$\mathcal J$ classes of $S$, such that $\cup_{i=1}^kJ_i$ is
an ideal of $S$ for all $1 \leq i \leq k$. We saw earlier that for
some $N \in \mathbb N$, we have $f^N=f^{N+1}$ for every $f \in
Cayley(S,J_1)$. Continuing by induction on the number of $\mathcal
J$ classes in $S$, and by applying the previous lemma, the result
follows.
\end{proof}

\section{Summary}

The results of this paper may be summed up as follows.
\begin{prop}
Let $S$ be a finite semigroup. The following are equivalent.
\begin{enumerate}
\item
$S$ is an aperiodic semigroup.
\item
$Cayley(S)$ is finite.
\item
$Cayley(S)$ is aperiodic.
\end{enumerate}
\end{prop}

\begin{proof}
That $(1)$ implies $(2)$ and $(1)$ implies $(3)$, are the main
results of this paper. Suppose that $(1)$ does not hold. Then $S$
contains a non trivial group $G$. By results of Silva and Steinberg
\cite{benpedro}, $Cayley(G)$ is a free semigroup on $|G|$
generators. Since $Cayley(G) \prec Cayley(S)$ we see that if $(1)$
is not true, then $(2)$ and $(3)$ are not true.
\end{proof}

\section{Acknowledgments}
The author would like to thank the anonymous referee for his careful reading and helpful comments, and the Bar-Ilan President's Fellowship for funding this research.

\end{document}